\documentclass[a4paper, 12pt]{article}
\usepackage{amssymb}
\usepackage{amsmath}
\usepackage{amsthm}
\usepackage[noadjust]{cite}
\usepackage[dvipdfm]{graphicx} 
\usepackage{verbatim,enumerate}
 \newtheorem{theorem}{Theorem}[section]
 \newtheorem*{theorem*}{Theorem}
 \newtheorem*{lemma*}{Lemma}
 \newtheorem{proposition}[theorem]{Proposition}
 
 \newtheorem{fact*}{Fact}
 \newtheorem{lemma}[theorem]{Lemma}
 
 \newtheorem{introtheorem}{Theorem}
   
\theoremstyle{definition}

 \newtheorem*{remark*}{Remark}

\numberwithin{equation}{section}
\usepackage[usenames]{color}

\newcommand{\R}{\boldsymbol{R}}
\newcommand{\Z}{\boldsymbol{Z}}
\newcommand{\C}{\boldsymbol{C}}

\newcommand{\rank}{\operatorname{rank}}

\renewcommand{\phi}{\varphi}

\newcommand{\hess}{\operatorname{Hess}}

\newcommand{\A}{\mathcal{A}}

\newcommand{\RR}{\mathcal{R}}
\newcommand{\ep}{\varepsilon}

\newcommand{\pmt}[1]{{\begin{pmatrix} #1  \end{pmatrix}}}

\newcommand{\trans}[1]{\vphantom{#1}^t\!#1}
\newcommand{\mycomment}[1]{}
\allowdisplaybreaks[2]
\newcommand{\addresslist}
{
\begin{tabular}{l}
(Y. Kabata)\\
School of Information and Data Sciences,\\
Nagasaki University, \\
Bunkyocho 1-14, Nagasaki, 852-8131, Japan\\
{\tt kabata@nagasaki-u.ac.jp}\\[3mm]
(K. Saji)\\
Department of Mathematics,\\
Graduate School of Science, \\
Kobe University, \\
Rokkodai 1-1, Nada, Kobe, 657-8501, Japan\\
{\tt saji@math.kobe-u.ac.jp}
\end{tabular}
}

\begin{document}
\title%[Criteria for sharksfin and deltoid]
{Criteria for sharksfin and deltoid singularities
from the plane into the plane
and their applications}
\author{Yutaro Kabata \and Kentaro Saji}
%\author{Yutaro Kabata}
%\address{
%School of Information and Data Sciences,
%Nagasaki University, 
%Bunkyocho 1-14, Nagasaki
%852-8131, Japan}
%\email{\tt kabata@nagasaki-u.ac.jp}
%\thanks{}

%\author{Kentaro Saji}
%\address{
%Department of Mathematics,
%Graduate School of Science, 
%Kobe University, 
%Rokkodai 1-1, Nada, Kobe 
%657-8501, Japan}
%\email{saji@math.kobe-u.ac.jp}
%\thanks{Partly supported by the
%JSPS KAKENHI Grants numbered 18K03301, 20K14312
%and the Japan-Brazil bilateral project JPJSBP1 20190103.}

%\subjclass[2020]{Primary 
%57R45; %Singularities of differentiable mappings in differential topolog
%Secondary 58K05 %Critical points of functions and mappings on manifold
%}

\date{\today}
\maketitle
\renewcommand{\thefootnote}{\fnsymbol{footnote}}
\footnote[0]{Keywords and Phrases: 
criteria, sharksfin, deltoid, Whitney umbrella, planar motions}
\footnote[0]{Partly supported by the
Japan Society for the Promotion of Science KAKENHI 
Grants numbered 18K03301, 20K14312
and the Japan-Brazil bilateral project JPJSBP1 20190103.}
\footnote[0]{Mathematics Subject Classification 2020 Primary 
57R45; %Singularities of differentiable mappings in differential topolog
Secondary 58K05 %Critical points of functions and mappings on manifold
}
\begin{abstract}
We give criteria for sharksfin and deltoid singularities
from the plane into the plane.
We also give geometric meanings for the conditions
in the criterion of a sharksfin.
As applications, we investigate such singularities
appearing on an orthogonal projection
of a Whitney umbrella, and a sharksfin appearing on 
planar motions with $2$-degrees of freedom.
\end{abstract}
\section{Introduction}
The singularities of maps from the plane into the plane
have been one of the fundamental subjects in the singularity theory
of the smooth maps.
Classification of singularities 
of maps from the plane into the plane has been investigated by
many researchers, 
and useful recognition criteria for 
main corank one singularities are given
(see for example \cite{bdupw,goryu,kabata, ohmaic, rieger, rieruas, sajipl,whitney}).
However,  as long as the authors know,
useful recognition criteria for corank two
maps have not been given in the literature.
In this paper, we give useful recognition criteria for
``sharksfin'' and ``deltoid'' singularities,
which are the most generic singularities of 
corank two maps from the plane into the plane:
\begin{introtheorem}\label{thm:cri}
Let\/ $f:(\R^2,0)\to(\R^2,0)$ be a map-germ satisfying\/
$\rank df_0=0$, and let\/
$\lambda$ be an identifier of singularities of $f$.
Let\/ $\lambda$ have a non-degenerate critical point
at\/ $0$, and let\/ $\eta_1,\eta_2$ be 
vector fields which satisfy that\/
$\eta_1(0)$ and\/ $\eta_2(0)$ are linearly independent
solutions of the Hesse quadric of\/ $\lambda$ at\/ $0$.
Then\/ $f$ is a sharksfin\/
$($respectively, deltoid\/$)$ at\/ $0$
if and only if\/
$\det \hess (\lambda)(0)<0$ $($respectively, $\det \hess (\lambda)(0)>0)$,
\begin{equation}\label{eq:cricond}
\det(\eta_1^2f,\eta_1^3f)(0)\ne0\quad\text{and}\quad
\det(\eta_2^2f,\eta_2^3f)(0)\ne0.
\end{equation}
\end{introtheorem}
Here, $\eta f$ stands for the directional derivative of
$f$ by the vector field $\eta$, and
$\eta^if=\eta(\eta^{i-1}f)$.
See Section \ref{sec:subpre} for the definitions
of
identifier of singularities
and
solutions of the Hesse quadric of a function at a non-degenerate
critical point.
We also give a geometric meaning of the condition 
\eqref{eq:cricond} as follows:
If an identifier of singularities has an 
index $1$ non-degenerate critical point at $0$, then
the singular set $S(f)$ consists of two transversal regular
curves, say $\gamma_i$ $(i=1,2)$.
The condition 
\eqref{eq:cricond} is then equivalent to
both $f\circ\gamma_i$ $(i=1,2)$ having $3/2$-cusps at $0$ 
(Proposition \ref{prop:gamma12cusp}).

As an application of the criteria, we investigate the
geometry of an orthogonal projection of a Whitney umbrella 
(Section \ref{sec:whit}) and 
singularities of planar motions with $2$-degrees of freedom 
(Section \ref{sec:plmo}).
Since maps from the plane into the plane
appear in several geometric situations
(see for example \cite{gibsonmarshxiang, kabata, OTflat, sajipl, WE}), 
the authors believe these criteria will work well 
if one wishes to find concrete conditions for
a given map to be a sharksfin or a deltoid.

See \cite{sajid4} for the criteria for 
$D_4^\pm$-singularities of fronts,
which are similar singularities of a sharksfin and a deltoid.

\section{Preliminaries and proof of criteria}\label{sec:pre}
\subsection{Preliminaries}\label{sec:subpre}
A map-germ $f:(\R^2,0)\to(\R^2,0)$ is called
a {\it sharksfin\/} (respectively, {\it deltoid\/})
if it is $\A$-equivalent to
the map-germ
$$
(u,v)\mapsto(uv,u^2+v^2+u^3)\quad
\big(\text{respectively,}\ (uv,-u^2+v^2+u^3)\big)
$$
at the origin $0$.
Here two map-germs
$f_i:(\R^2,0)\to(\R^2,0)$ $(i=1,2)$ are said to be
{\it $\A$-equivalent\/} if there exist
a diffeomorphism-germ $\phi$ on the source space of $f_1$ and
a diffeomorphism-germ $\Phi$ on the target space of $f_2$
such that $\Phi\circ f_1=f_2\circ\phi$ holds.
We note that the map-germ $
(u,v)\mapsto(uv,u^2+v^2+u^3)$ is $\A$-equivalent to 
$(u,v)\mapsto(u^2+v^3,v^2+u^3)$, and
some of the literature uses this form for a sharksfin.
A fundamental classification of map-germs from the plane into the plane 
is given in \cite{whitney,rieger}. 
In \cite{rieruas}, it is shown that
a sharksfin and a deltoid are the only singularities 
of rank zero and codimension equal to or less than four.
See \cite{kabata,sajipl,suyak,whitney} for recognition criteria
for fundamental rank one singularities.

Let $f:(\R^2,0)\to(\R^2,0)$ be a map-germ and let
$\rank df_0=0$.
Let $\lambda$ be a non-zero functional multiple
of the Jacobian of $f$.
Since the zeros of $\lambda$ are the set of singular points of $f$,
we call the function $\lambda$ an {\it identifier of singularities}.
By the assumption $\rank df_0=0$, the function
$\lambda$ has a critical point at $0$.

To state the criteria, we define 
vectors at a non-degenerate critical point of a function.
Let $h:(\R^2,0)\to(\R,0)$ be a function which has
a non-degenerate critical point at $0$.
The {\it solution of the Hesse quadric\/} of $h$ at $0$
is a non-zero vector $(x,y)$ which satisfies
$$
\left(\pmt{x\ y}\hess (h)(0)\pmt{x\\y}
=\right)h_{uu}(0)x^2+2h_{uv}(0)xy+h_{vv}(0)y^2=0,
$$
where 
$(\cdot)_u=\partial(\cdot)/\partial u$,
$(\cdot)_v=\partial(\cdot)/\partial v$
and
$\hess(h)(0)$ is the Hesse matrix of $h$ at $0$.
If $\det \hess (h)(0)>0$, then there
exist two $\C$-linearly independent $\C$-valued
vectors $v_1,v_2$.
If $\det \hess (h)(0)<0$, then there
exist two $\R$-linearly independent $\R$-valued
vectors $v_1,v_2$, and these are
the tangent vectors of two branch curves of
the zero set of $h$ at $0$.
For a diffeomorphism $\phi:(\R^2,0)\to(\R^2,0)$,
since the Hesse matrix of $h\circ\phi^{-1}$ is
$\trans{J_\phi(0)^{-1}}\hess (h)(0)J_\phi(0)^{-1}$,
and
$$
\pmt{x\ y}\hess (h)(0)\pmt{x\\y}=
\trans{\left(J_\phi(0)\pmt{x\\ y}\right)}\,
\hess (h\circ \phi^{-1})(0)\,J_\phi(0)\pmt{x\\y},
$$
the two directions defined by
solutions of the Hesse quadric of $h$ at $0$
do not depend on the choice of the coordinate system
on the source space.
Here, $\trans{(~)}$ stands for the matrix transposition,
and $J_\phi$ stands for the Jacobi matrix of $\phi$.
We note that if 
solutions of the Hesse quadric $\eta_i$ $(i=1,2)$ are complex vectors,
then the left-hand sides of \eqref{eq:cricond} may be
complex numbers.

\subsection{Proof of Theorem \ref{thm:cri}}
We give a proof of Theorem \ref{thm:cri}.
We first simplify the expression of a rank zero-germ 
$f:(\R^2,0)\to(\R^2,0)$ by
coordinate changes.
Keeping in mind that we will investigate geometric meanings,
we restrict to using only particular coordinate changes
in the next lemma.

\begin{lemma}\label{lem:coordchange010}
If\/ $\rank df_0=0$ and\/ 
$0$ is a non-degenerate critical point of\/ $\lambda$,
then there exist a positive local coordinate system\/ $(u,v)$ near\/
$0$, and 
an orientation preserving isometry\/ 
$M$ on\/  $(\R^2,0)$, namely, $M\in {\rm SO}(2)$, 
such that
$$
M\circ f(u,v)=(uv,O(2)).
$$
\end{lemma}
Here, $O(i)$ stands for the terms whose degrees are greater
than or equal to\/ $i$, namely, $O(i)$ is an element of
$$
\left\{f\in C^\infty(\R^2,0)\,\Big|\, 
\sum_{a+b\leq i-1}\dfrac{\partial^{a+b}f}{\partial u^a\partial v^b}(0)=0
\right\}.
$$
A coordinate system is said to be {\it positive\/}
if it has the same orientation to the standard $\R^2$.
\begin{proof}
Firstly we show that we may assume
$0$ is a non-degenerate critical point
with index $1$ of $f_1$, where $f=(f_1,f_2)$.
If $f_1=O(3)$ or $f_2=O(3)$,
then the Jacobian of $f$ is $O(3)$. In particular,
$0$ will never be a non-degenerate
critical point of an identifier of singularities.
Thus, we may assume $(f_1)_{uu}(0)\ne0$ by a suitable rotation
on the target and by choosing a positive coordinate system on the source.
By a rotation on the target,
$f$ can be written as
$f=(f_1,b_{11}uv+b_{02}v^2)+O(3)$.
If $b_{11}\ne0$, then the claim is proved.
We assume $b_{11}=0$.
Since $f_2\ne O(3)$, it holds that $b_{02}\ne0$.
Thus $f$ can be written as
$f=(a_{20}u^2+2a_{11}uv+a_{02}v^2,b_{02}v^2)+O(3)$.
If $a_{20}a_{02}-a_{11}^2<0$, then the claim is proved,
and if $a_{20}=0$, then $\lambda$ degenerates.
So, we may assume $a_{20}a_{02}-a_{11}^2>0$ and $a_{20}\ne0$.
By taking the rotation by degree $\theta$ on the target,
the first component of $f$ can be written as
$a_{20}\cos\theta u^2+2a_{11}\cos\theta uv+(a_{02}\cos\theta-b_{02}\sin\theta)v^2$.
One can find an angle $\theta$ such that 
$\cos\theta(\cos\theta(a_{20}a_{02}-a_{11}^2)-\sin\theta a_{20}b_{02})<0$.
This shows the claim.

Thus we may assume that $0$ is a non-degenerate critical point
with index $1$ of $f_1$.
By the Morse lemma, there exists a positive coordinate system $(u,v)$ such
that $f$ can be written as
$f=(\pm uv,O(2))$. Taking a $\pi$-rotation on the target if
necessary, this shows the assertion.
\end{proof}
By Lemma \ref{lem:coordchange010},
we see $f$ is $\A$-equivalent to
\begin{equation}\label{eq:fchange005}
f(u,v)=
\Big(uv,\ep \dfrac{u^2}{2}+\dfrac{v^2}{2}+
\sum_{i+j=3}\dfrac{a_{ij}u^iv^j}{i!j!}
+O(4)\Big) \ (\ep=\pm1).
\end{equation}
\begin{lemma}\label{lem:condo3}
Let\/ $f$ be a map-germ of the form\/
\eqref{eq:fchange005}.
When\/ $\ep=1$, then\/ $f$ is a sharksfin if and only if\/
$(a_{30}-3a_{21}+3a_{12}-a_{03})(a_{30}+3a_{21}+3a_{12}+a_{03})\ne0$.
When\/ $\ep=-1$, then\/ $f$ is deltoid if and only if\/
$a_{30} - 3 a_{12} \ne0$ or\/ $3 a_{21} - a_{03}\ne0$.
\end{lemma}
\begin{proof}
Following \cite[Propositions 2.1.1 and 2.2.1]{rieruas}, 
we will give a proof of this lemma.
By the coordinate change
$u=u_1+a_{12} u_1^2/2-\ep a_{21} u_1 v_1/2$, 
$v=v_1-a_{12} u_1 v_1/2+\ep a_{21} v_1^2/2$,
we see that
\begin{equation}\label{eq:fchange010}
f(u_1,v_1)=
\Big(u_1 v_1,\dfrac{\ep u_1^2}{2}+\dfrac{v_1^2}{2}
+\dfrac{\tilde{a}_{30}u_1^3}{6}+\dfrac{\tilde a_{03}v_1^3}{6} \Big)
+(O(4),O(4)),
\end{equation}
where $\tilde a_{30}=a_{30}+3 \ep a_{12}$
and
$\tilde a_{03}=a_{03}+3 \ep a_{21}$.
Let $f$ be written as in \eqref{eq:fchange010},
and $\ep=1$.
By the coordinate change
\begin{align*}
u_1=&u_2-v_2+(\tilde{a}_{03}-2 \tilde{a}_{30}) u_2^2/12
+\tilde{a}_{30} u_2 v_2/4- (\tilde{a}_{03}+2 \tilde{a}_{30}) v_2^2/12,\\
\quad
v_1=&u_2+v_2
+(\tilde{a}_{30}-2\tilde{a}_{03}) u_2^2/12 
-\tilde{a}_{03} u_2 v_2/4
-(2\tilde{a}_{03}+\tilde{a}_{30}) v_2^2/12
\end{align*}
and $\Phi(X,Y)=(X+Y,-X+Y)/2$,
we see
\begin{equation}\label{eq:fchange020}
\Phi\circ f(u_2,v_2)=
\Big(u_2^2 + 
\dfrac{(\tilde{a}_{03} - \tilde{a}_{30}) v_2^3}{12}, 
v_2^2+
\dfrac{(\tilde{a}_{03} + \tilde{a}_{30}) u_2^3}{12}\Big)
+(O(4),O(4)).
\end{equation}
Since a sharksfin is
$3$-determined (\cite[Proposition 2.1.1]{rieruas}),
$f$ in the form \eqref{eq:fchange005}
is a sharksfin if and only if
$(\tilde{a}_{03} - \tilde{a}_{30})(\tilde{a}_{03} + \tilde{a}_{30}) \ne0$,
namely,
$(a_{30}-3a_{21}+3a_{12}-a_{03})(a_{30}+3a_{21}+3a_{12}+a_{03})\ne0$.

Next, we show the case of a deltoid.
Let $f$ be written as in \eqref{eq:fchange010}
and $\ep=-1$.
Let $p(x)$ be the polynomial 
$$
p(x)=\tilde{a}_{03} x^5- 5 \tilde{a}_{30} x^4- 10 \tilde{a}_{03} x^3
+ 10 \tilde{a}_{30} x^2 + 5 \tilde{a}_{03} x -\tilde{a}_{30},
$$
and let $x_0$ be one of the solutions of $p(x)=0$.
We consider the coordinate change defined by
\begin{align*}
u_1=&
\dfrac{1}{96}
\Big(
96 u_2 -96x_0v_2
+u_2^2 (-5 \tilde{a}_{30} (-3+x_0^2)+\tilde{a}_{03} x_0 (5+x_0^2))\\
&\hspace{15mm}
+u_2v_2(-\tilde{a}_{03}  x_0^2 (-7+x_0^2)+5 \tilde{a}_{30} x_0 (-3+x_0^2))\\
&\hspace{25mm}
+v_2^2 (\tilde{a}_{30}  (-9+35 x_0^2)
      +x_0 (\tilde{a}_{03} (13-7 x_0^2)))\Big),\\
v_1=&
\dfrac{1}{96}
\Big(
96u_2 x_0+96v_2
- u_2^2 x_0 (-5 \tilde{a}_{30} (-3+x_0^2)+\tilde{a}_{03} x_0 (5+x_0^2))\\
&\hspace{15mm}
-4 u_2 v_2 (\tilde{a}_{30} (3-5 x_0^2)+\tilde{a}_{03} x_0 (5+x_0^2))\\
&\hspace{25mm}
+v_2^2 x_0 (\tilde{a}_{30} (57-35 x_0^2)+\tilde{a}_{03} x_0 (-61+7 x_0^2))\Big),
\end{align*}
and we consider
$$
\Phi(X,Y)
=
\Big(
\dfrac{(1 - x_0^2) X + 2 x_0 Y}{(1 + x_0^2)^2}, 
-\dfrac{2(x_0^2-1)(2 x_0 X +(x_0^2 -1)Y)}
{(1 + x_0^2)^2 (x_0^2-1)}\Big).
$$
Then $\Phi\circ f$ is
\begin{align}
&
\bigg(
u_2 v_2,
-u_2^2+v_2^2
+
\dfrac{(\tilde{a}_{03}^2 + \tilde{a}_{30}^2)(1 - 10 x_0^2 + 5 x_0^4)}
{3 \tilde{a}_{03} (1+x_0^2)^2}v_2^3\bigg)\nonumber\\
&+
\dfrac{p(x_0)}{96 (1+x_0^2)^2}
\bigg(-2x_0u_2^3
      +3(x_0^2-1)u_2^2 v_2
      +6x_0u_2 v_2^2
      +(9+7 x_0^2)v_2^3,\\
\label{eq:fchange120}
&\hspace{5mm}2\big(
 (x_0^2-1)u_2^3
+6x_0u_2^2 v_2
-3(x_0^2-1)u_2 v_2^2\nonumber\\
&\hspace{5mm}2
+2(8 \tilde{a}_{30}/\tilde{a}_{03} + 7 x_0)v_2^3
\big)\bigg)+(O(4),O(4)).
\nonumber
\end{align}
The solutions of $5 x^4-10 x^2 + 1=0$ are
$x_1=\pm (1 \pm 2/\sqrt{5})^{1/2}$,
and we see
$5p(x_1)=\pm 16\tilde{a}_{03}(10\pm 22/\sqrt{5})^{1/2}$
is not zero if $\tilde{a}_{03}\ne0$.
Since a deltoid is
$3$-determined (\cite[Proposition 2.2.1]{rieruas}),
$f$ in the form \eqref{eq:fchange005}
is a deltoid if and only if
$\tilde{a}_{30}\ne0$ or $\tilde{a}_{03}\ne0$,
namely,
$a_{30} - 3 a_{12} \ne0$ or $3 a_{21} - a_{03}\ne0$.
\end{proof}
\begin{lemma}\label{lem:indepvf}
Let\/ $0$ be a non-degenerate critical point of an identifier
of singularities\/ $\lambda$ of\/ $f:(\R^2,0)\to(\R^2,0)$, and let\/
$(\eta_{11},\eta_{12})$, $(\eta_{21},\eta_{22})$ 
$($possibly complex\/$)$ be two linearly independent solutions 
of the Hesse quadric
of\/ $\lambda$ at\/ $0$.
Let\/ $\eta_i$ be a vector field satisfying
$$
\eta_i(0)=\eta_{i1}\partial_u+\eta_{i2}\partial_v\quad
(i=1,2).$$
Then the condition
$$
\det(\eta_1^2f,\eta_1^3f)(0)\ne0\quad\text{and}\quad
\det(\eta_2^2f,\eta_2^3f)(0)\ne0
$$
does not depend on the choice of\/ $\eta_1,\eta_2$.
\end{lemma}
\begin{proof}
The condition clearly does not depend on the choice of
the coordinate system on the source space, and
the condition does not change by a linear coordinate
transformation on the target space.
As we remarked just after Lemma \ref{lem:coordchange010},
we may assume $f$ is written as in
\eqref{eq:fchange005}.
Then the identifier of singularities
is $-\ep u^2+v^2+O(3)$.
Since the condition does not change under
a constant multiplication of 
the solutions of the Hesse quadric $\eta_1$ and $\eta_2$
of an identifier of singularities,
we may assume
$\eta_1=(1,1),\eta_2=(1,-1)$ if $\ep=1$,
and
$\eta_1=(1,i),\eta_2=(1,-i)$ if $\ep=-1$, 
where $i=\sqrt{-1}$.
By a direct calculation,
\begin{align*}
\xi^2f&=
(2 \xi_1 \xi_2,\ep \xi_1^2+\xi_2^2),\\
\xi^3f&=
\Big(3 (\xi_2^2 (\xi_1)_v+\xi_1 \xi_2 ((\xi_2)_v+(\xi_1)_u)
+\xi_1^2 (\xi_2)_u),\\
 &\hspace{15mm}
a_{30} \xi_1^3+3 a_{21} \xi_1^2 \xi_2+3 a_{12} \xi_1 \xi_2^2
+a_{03} \xi_2^3\\
&\hspace{20mm}
+3 \ep \xi_1^2 (\xi_1)_u
+3 \ep \xi_1 \xi_2 (\xi_1)_v
+3 \xi_1 \xi_2 (\xi_2)_u
+3 \xi_2^2 (\xi_2)_v
\Big),
\end{align*}
holds at $u=v=0$,
where $\xi=(\xi_1,\xi_2)$ is a vector field, and
\begin{equation}\label{eq:conddet}
\dfrac{1}{2}\det(\xi^2f,\xi^3f)(0)
=\left\{
\begin{array}{l}
 a_{30} + 3 a_{21} + 3 a_{12} + a_{03}  \quad(\ep=1,\xi(0)=\eta_1=(1,1))\\
-a_{30} + 3 a_{21} - 3 a_{12} + a_{03} \quad(\ep=1,\xi(0)=\eta_2=(1,-1))\\
a_{30}i - 3 a_{21}- 3 i a_{12} +a_{03}  \quad(\ep=-1,\xi(0)=\eta_1=(1,i))\\
-a_{30}i - 3 a_{21}+ 3 i a_{12} +a_{03}  \quad(\ep=-1,\xi(0)=\eta_2=(1,-i)).\\
\end{array}
\right.
\end{equation}
The right-hand side of \eqref{eq:conddet} depends only on the value 
of $\xi(0)$. This shows the assertion.
\end{proof}
\begin{proof}[Proof of Theorem\/ {\rm \ref{thm:cri}}]
The sufficiency follows by the independence of
the choice of coordinate systems and vector fields.
We show the necessity.
We assume the assumption of Theorem \ref{thm:cri}.
Since 
the condition does not depend on the choice of 
coordinate systems and vector fields,
we may assume $f$ is written as in \eqref{eq:fchange005}.
Then by \eqref{eq:conddet} in the proof of Lemma \ref{lem:indepvf},
the condition \eqref{eq:cricond} in Theorem \ref{thm:cri}
is equivalent to
$
( a_{30} + 3 a_{21} + 3 a_{12} + a_{03})
(-a_{30} + 3 a_{21} - 3 a_{12} + a_{03})\ne0$ when $\ep=1$,
and
$a_{30}- 3 a_{12}\ne0$ or $3 a_{21}-a_{03}\ne0$ when $\ep=-1$.
By Lemma \ref{lem:condo3}, we have the assertion.
\end{proof}

\section{Geometry of a sharksfin and a deltoid}
\subsection{Geometric meanings of a criterion for a sharksfin}
We here give a geometric interpretation of condition \eqref{eq:cricond}.
Let $f:(\R^2,0)\to(\R^2,0)$ be a map-germ with
$\rank df_0=0$, and let an identifier of singularities
$\lambda$ have an index one critical point
at $0$.
Then the set $\lambda^{-1}(0)$ consists of
images of two transversal regular curves passing through $0$.
We set these curves as $\gamma_i:(\R,0)\to(\R^2,0)$.
A curve-germ $c:(\R,0)\to(\R^2,0)$ at $0$ is a $3/2$-{\it cusp\/}
if it is $\A$-equivalent to $t\mapsto (t^2,t^3)$.
It is well-known that $c:(\R,0)\to(\R^2,0)$ is a $3/2$-cusp
if and only if $c'(0)=0$, and $\det(c''(0),c'''(0))\ne0$.
Let $c:(\R,0)\to(\R^2,0)$ be a curve and $0$ a $3/2$-cusp.
The {\it cuspidal direction\/} of $c$ at $0$ is the direction
defined by $c''(0)$.
The cuspidal direction bisects the cusp.
Then the following proposition holds.
\begin{proposition}\label{prop:gamma12cusp}
Under the above setting, 
$f$ at\/  $0$ is a sharksfin if and only if\/ 
$f\circ \gamma_i$ $(i=1,2)$ at\/ $0$ are both\/ $3/2$-cusps.
\end{proposition}
\begin{proof}
Since the condition and the assertion do not depend on the choice
of the coordinate systems, we may assume that $f$ is written as
\eqref{eq:fchange005} with $\ep=1$.
Then we may assume that
$\gamma_1=\gamma_+=(t,a_+(t)), \gamma_2=\gamma_-=(t,a_-(t))$ ($a_+(0)=a_-(0)=0$).
Since $\lambda(t,a_\pm(t))=0$,
we have $a'_\pm(0)=\pm1$.
We set $\hat\gamma_i(t)=f\circ \gamma_i(t)$.
Then we see
$\hat\gamma_\pm''(0)=2(\pm1,1)$,
$\hat\gamma_1'''(0)=(3a''_\pm (0),3a''_\pm (0)\pm a_{30}+3a_{21}\pm 3a_{12}+a_{03})$.
Thus $$
\dfrac{1}{2}\det\big(\hat\gamma_\pm''(0),\hat\gamma_\pm'''(0)\big)=
a_{30}\pm 3a_{21}+3a_{12}\pm a_{03}.
$$
By \eqref{eq:conddet}, we have the assertion.
\end{proof}

\subsection{An ${\rm SO}(2)$-normal form}
We give a simplified form of a given rank zero germ
by using diffeomorphisms on the source and 
isometries on the target space.
Since coefficients of such forms are differential geometric
invariants, this is convenient for studying the differential geometry
of singularities. See 
\cite{bw,FH2,fukui,hhnuy,hnsuydual,MS,axial,OTflat,sajisw,ttohoku,WE} 
such studies, for example.
Two map-germs 
$f_i:(\R^2,0)\to(\R^2,0)$ are said to be 
{\it $\RR_+\times {\rm SO}(2)$-equivalent\/}
if there exist an orientation preserving
diffeomorphism-germ $\phi:(\R^2,0)\to(\R^2,0)$
and an orientation preserving isometry
$M$ on $(\R^2,0)$, namely, $M\in {\rm SO}(2)$, 
such that
$M\circ f_1=f_2\circ \phi$ holds.
\begin{proposition}\label{prop:normal}
Let\/  $f:(\R^2,0)\to(\R^2,0)$ be a map-germ with\/ 
$\rank df_0=0$, and let an identifier of singularities\/ 
$\lambda$ have a non-degenerate critical point
at\/  $0$.
Then\/  $f$ is\/  $\RR_+\times {\rm SO}(2)$-equivalent
to the germ
\begin{align}
&\left(uv,\dfrac{\ep_1a_{20} u^2}{2}+\dfrac{\ep_2 a_{20}v^2}{2}
+\dfrac{a_{30}u^3}{6}+
\dfrac{a_{03}v^3}{6}\right)+(O(4),O(4)),\label{eq:normald4}\\
&\hspace{10mm} \quad(a_{20}>0, (\ep_1,\ep_2)\in\{(1,1),(1,-1),(-1,1)\}).
\nonumber
\end{align}
\end{proposition}
We show this proposition by taking $\RR_+\times {\rm SO}(2)$-equivalent germs
step by step.
First we show the following lemma.
\begin{lemma}\label{lem:normal0100}
Let\/  $f:(\R^2,0)\to(\R^2,0)$ be a map-germ satisfying the condition
in Proposition\/  {\rm \ref{prop:normal}}.
Then\/  $f$ is\/  $\RR_+\times {\rm SO}(2)$-equivalent
to the germ
\begin{equation}\label{eq:normal0100}
\left(uv,\dfrac{\ep_1a_{20} u^2}{2}+a_{11}uv+\dfrac{\ep_2a_{20}v^2}{2}+O(3)\right)
 \quad(a_{20}\ne0, \ep_1,\ep_2\in\{1,-1\}).
\end{equation}
\end{lemma}
\begin{proof}
By Lemma \ref{lem:coordchange010},
we may assume $f$ is written as
$f(u,v)=(uv,a_{20}u^2/2+a_{11}uv+a_{02}v^2+O(3))$.
Since $0$ is a non-degenerate critical point of $\lambda$,
we see $a_{20}a_{02}\ne0$.
Then by setting 
$u=|a_{02}/a_{20}|^{1/4}u_1,v=|a_{20}/a_{02}|^{1/4}v_1$,
we see the assertion.
\end{proof}
\begin{lemma}\label{lem:normal0200}
Let\/  $f:(\R^2,0)\to(\R^2,0)$ be a map-germ satisfying the assumption
of Proposition\/  {\rm \ref{prop:normal}}.
Then\/  $f$ is\/  $\RR_+\times {\rm SO}(2)$-equivalent
to the germ
\begin{equation}\label{eq:normal0200}
\left(b_{11}uv,\dfrac{a_{20} u^2}{2}+\dfrac{a_{02}v^2}{2}\right)
+(O(3),O(3))
 \quad(b_{11}a_{20}a_{02}\ne0).
\end{equation}
\end{lemma}
\begin{proof}
By Lemma \ref{lem:normal0100},
we may assume $f$ is written as \eqref{eq:normal0100}.
If $a_{11}=0$, then no proof is needed. We assume $a_{11}\ne0$.
We set $u=u_1+v_1$, $v=cu_1+dv_1$ $(c,d\in\R)$, and
$M$ is the rotation matrix by degree $\theta$.
Then $M\circ f(u,v)$ is
\begin{align}
\label{eq:normal0250}
&\Big(
(2 c \cos\theta - (a_{20} + 2 a_{11} c + a_{20} c^2) \sin\theta)\dfrac{u^2}{2} 
+*uv\\
&\hspace{10mm}+
(2 d \cos\theta - (a_{20} + 2 a_{11} d + a_{20} d^2) \sin\theta)\dfrac{v^2}{2},
\nonumber\\
&\hspace{1mm}
*u^2
+((a_{20} + a_{20} c d + a_{11} (c + d)) \cos\theta + (c + d) \sin\theta)uv
+*v^2\Big)+(O(3),O(3)),
\nonumber
\end{align}
where $*$ stands for a real number which will not be needed
in later calculations.
To show the lemma, we need to solve the equation
\begin{equation}\label{eq:normal0300}
\left\{
\begin{array}{l}
2 c \cos\theta - (a_{20} + 2 a_{11} c + a_{20} c^2) \sin\theta =0,\\
2 d \cos\theta - (a_{20} + 2 a_{11} d + a_{20} d^2) \sin\theta=0,\\
(a_{20} + a_{20} c d + a_{11} (c + d)) \cos\theta + (c + d) \sin\theta=0
\end{array}
\right.
\end{equation}
with respect to $c,d,\theta$.
If $\sin\theta=0$, then $c=d=0$, and the third equation of
\eqref{eq:normal0300} is $a_{20}=0$.
Thus $\sin\theta=0$ is not a solution.
We assume $\sin\theta\ne0$.
Noticing the function $\cos\theta/\sin\theta$ takes 
values in $\R$, we set $t=\cos\theta/\sin\theta$.
Then \eqref{eq:normal0300} is equivalent to
\begin{align}
2 c t - (a_{20} + 2 a_{11} c + a_{20} c^2)  =0,\label{eq:normal0400}\\
2 d t - (a_{20} + 2 a_{11} d + a_{20} d^2) =0,\label{eq:normal0500}\\
(a_{20} + a_{20} c d + a_{11} (c + d)) t + c + d =0.\label{eq:normal0600}
\end{align}
Note that the equations \eqref{eq:normal0400} and \eqref{eq:normal0500}
are the same.
Thus the solutions $c,d$ are the two solutions of the equation 
\begin{equation}\label{eq:normal0650}
-a_{20} x^2 +2(-a_{11} + t)x-a_{20}=0
\end{equation}
with respect to $x$.
We set $c,d$ are these two solutions satisfying $d-c>0$,
where we will see \eqref{eq:normal0650} has distinct real roots.
Hence $c+d=2(-a_{11}+t)/a_{20}$ and $cd=1$.
Substituting this into \eqref{eq:normal0600},
we obtain one of the solutions
\begin{equation}\label{eq:normal0700}
t=\dfrac{-(1-a_{11}^2+a_{20}^2)+\sqrt{(1-a_{11}^2+a_{20}^2)^2+4a_{11}^2}}{2a_{11}},
\end{equation}
where $(1-a_{11}^2+a_{20}^2)^2+4a_{11}^2>0$ is obvious.
Substituting this into \eqref{eq:normal0400} and \eqref{eq:normal0500},
we obtain solutions $c$ and $d$.
Now we see that the equation 
\eqref{eq:normal0650} has two distinct real roots under
\eqref{eq:normal0700}.
However this can be easily shown from
$$
(-a_{11}+t)^2-a_{20}^2>\dfrac{(1-a_{11}^2+a_{20}^2)^2+4a_{11}^2}{4a_{11}^2}.
$$
Finally, 
since $0$ is a non-degenerate critical point of $\lambda$,
neither the coefficients of 
$uv$ in the first component nor those of 
$u^2,v^2$ in the second component
vanish.
\end{proof}
\begin{proof}[Proof of Proposition\/  {\rm \ref{prop:normal}}]
By Lemma \ref{lem:normal0200},
we may assume $f$ is written as in \eqref{eq:normal0200}.
By the Morse lemma, there exists a coordinate change
$u=au_1+bv_1+O(2),v=cu_1+dv_1+O(2)$ such that 
the first component of $f$ in \eqref{eq:normal0200}
written by the coordinate system $(u_1,v_1)$
is $u_1v_1$.
Then we have $ac=bd=0$, and we may assume $b=c=0$,
since the case $a=d=0$ is similar.
If the coordinate change reverses the orientation, we
compose with $(u,v)\mapsto (-u,v)$.
This implies that we may assume $f$ is written as
$$
%\begin{equation}\label{eq:normal0900}
\left(uv,\dfrac{a_{20} u^2}{2}+\dfrac{a_{02}v^2}{2}
+
\dfrac{a_{30}u^3}{6}+
\dfrac{a_{21}u^2v}{2}+
\dfrac{a_{12}uv^2}{2}+
\dfrac{a_{03}v^3}{6}+O(3)
\right)
 \quad(a_{20}a_{02}\ne0).
$$
If the first component is $-uv$, then we take a $\pi$-rotation
on the target. 
By the coordinate change
$$
u=u_{1}+\dfrac{a_{12} u_{1}^2}{2 a_{02}}-\dfrac{a_{21} u_{1} v_{1}}{2 a_{20}},\quad
v=v_{1}-\dfrac{a_{12} u_{1} v_{1}}{2 a_{02}}+\dfrac{a_{21} v_{1}^2}{2 a_{20}},
$$
$f$ is written as 
$$
\left(u_1v_1,
\dfrac{a_{20} u_1^2}{2}+\dfrac{a_{02}v_1^2}{2}+
\dfrac{\tilde{a}_{30}u_1^3}{6}+\dfrac{\tilde{a}_{03}v_1^3}{6}\right)+(O(4),O(4)).
$$
By setting 
$u_1=|a_{02}/a_{20}|^{1/4}u_2,v_1=|a_{20}/a_{02}|^{1/4}v_2$,
and taking $(u,v)\mapsto(v,-u)$ and $\pi$-rotation on the target
if $a_{20}<0$ and $a_{02}<0$, 
we see the assertion.
\end{proof}
We remark that the number of coefficients of the terms 
whose degrees are less than or equal to $3$ in the
form \eqref{eq:normald4} in Proposition \ref{prop:normal} is $3$.
However, coefficients in the first component remain.
When considering higher degree terms, it should be remarked that
the form \eqref{eq:normal0100} in Lemma \ref{lem:normal0100} is convenient,
since its first component is just $uv$.
\subsection{Geometric meaning of coefficients of the ${\rm SO}(2)$-normal form}
Let $f:(\R^2,0)\to(\R^2,0)$ at $0$ be a sharksfin.
In Proposition \ref{prop:gamma12cusp}, we showed that the images of 
two branch curves of $S(f)$ are both $(2,3)$-cusps.
Here we give an interpretation of the coefficients 
$a_{20},a_{30},a_{03}$ in the ${\rm SO}(2)$-normal form \eqref{eq:normald4}
by geometries of these $(2,3)$-cusps.
As in the proof of Proposition \ref{prop:gamma12cusp},
we denote two branch curves of $S(f)$ by $\gamma_i$ $(i=1,2)$
and set $\hat\gamma_i=f\circ\gamma$.
By the proof of Proposition \ref{prop:gamma12cusp},
the cuspidal directions of $\hat\gamma_i$ $(i=1,2)$ at $0$ are 
linearly independent.
Thus the angle between two cuspidal directions of $\hat\gamma_i$ $(i=1,2)$
is a geometric invariant of $f$.
On the other hand, let $c:(\R,0)\to(\R^2,0)$ at $0$ be a $3/2$-cusp.
The {\it cuspidal curvature\/} of $c$ at $0$ is
defined by
$$
\kappa_{cusp}(c)
=
\dfrac{\det(c''(0), c'''(0))}
{|c''(0)|^{5/2}}.
$$
The cuspidal curvature measures the wideness of the $3/2$-cusp.
See \cite{suyo,umecusp} for details.
Since $\hat\gamma_i$ $(i=1,2)$ have $3/2$-cusps at $0$,
one can compute the cuspidal curvatures.
We may assume that $f$ is written as in
\eqref{eq:normald4}.
Then in the same notation and by the same arguments as 
in the proof of Proposition \ref{prop:gamma12cusp},
we have
$a_\pm=\pm t+(a_{30}-a_{03})t^2/(4a_{20})+O(3)$,
and
$$
\hat\gamma_\pm''(0)=(\pm 2,2 a_{20}),\quad
\hat\gamma_\pm'''(0)=
\left( \dfrac{3 (\pm a_{30}- a_{03})}{2 a_{20}},
\dfrac{5 a_{30}\mp a_{03}}
{2}\right).
$$
Thus 
\begin{equation}\label{eq:kappacusp}
\kappa_{cusp}(\hat\gamma_\pm)
=
\dfrac{2 (\pm a_{30}+a_{03})}{(4+4 a_{20}^2)^{5/4}},
\end{equation}
and the angle $\theta_\gamma$ between the two cuspidal directions of 
$\hat\gamma_i$ $(i=1,2)$
is
\begin{equation}\label{eq:cuspangles}
\cos^{-1}\left(
\dfrac{|a_{20}^2-1|}
{a_{20}^2+1}
\right).
\end{equation}
By \eqref{eq:kappacusp}, \eqref{eq:cuspangles} and $a_{20}>0$,
together with
Proposition \ref{prop:normal},
the cuspidal curvatures of $\hat\gamma_\pm$
and the angle between the two cuspidal directions
determines the sharksfin up to three degrees,
namely, 
$\kappa_{cusp}(\hat\gamma_\pm)$ and $\theta_\gamma$
determines the $\RR_+\times {\rm SO}(2)$-class of sharksfins up to
$3$-jets.
It is known that $f(S(f))$ determines the $\RR$-class
of a sharksfin, since it is a critical normalization
\cite{dupgafwil,wirth}.
On the other hand, a deltoid is not a critical normalization,
and we cannot find such invariants in terms of $f(S(f))=\{0\}$.
\subsection{Projection of a Whitney umbrella}\label{sec:whit}
Let $f:(\R^2,0)\to(\R^3,0)$ be a {\it Whitney umbrella\/}
or equivalently, a {\it cross cap}, 
namely, it is $\A$-equivalent to the germ
$(u,v)\mapsto(u,v^2,uv)$.
A line generated by $df_0(T\R^2)$ is called the {\it tangent line\/}
of $f$, and a plane $P$ perpendicular to the tangent line
is called the {\it normal plane}.
Let $\pi:\R^3\to P$ be the orthogonal projection.
It is known that if $f$ is a Whitney umbrella, then
$\pi\circ f$ is a sharksfin or a deltoid, generically.
More precisely, 
if $f:(\R^2,0)\to(\R^3,0)$ is a Whitney umbrella, then
there exist a coordinate system $(u,v)$ and a rotation $\Phi$
on $\R^3$ such that
\begin{equation}\label{eq:whitney}
f(u,v)=\left(
u,uv+\dfrac{c_3v^3}{6},
\sum_{i+j=2}^3 \dfrac{d_{ij}u^iv^j}{i!j!}\right)
+(0,O(4),O(4))\quad(c_3,d_{02}>0,d_{ij}\in\R).
\end{equation}
See \cite{WE} or \cite{FH2}.
A Whitney umbrella is said to be {\it elliptic\/} 
(respectively, {\it hyperbolic\/}) if 
$d_{20}>0$ (respectively, $d_{20}<0$).
See \cite{bw,FH2,hhnuy,ttohoku,WE} for the geometric meanings
of other coefficients.
Considering the orthogonal projection with respect to the tangent line
into the normal plane,
a rank zero singular point appears.
A saharksfin (respectively, deltoid)
appears on the projection of an elliptic (respectively, hyperbolic)
Whitney umbrella.
We give a precise condition for this fact in terms of the
coefficients $c_3,d_{ij}$ $(i+j=2,3)$ in \eqref{eq:whitney}, namely,
geometric information of the Whitney umbrella. 
\begin{theorem}\label{eq:whitproj}
Let\/ $f$ be a Whitney umbrella written in the form\/
\eqref{eq:whitney} with\/ $d_{20}\ne0$.
Let\/ $\pi:\R^3\to \R^2$ be the orthogonal 
projection\/ $(X_1,X_2,X_3)\mapsto(X_2,X_3)$.
Then\/ $\pi\circ f$ at\/ $0$ is a sharksfin if and only if\/
$d_{20}>0$ and
\begin{equation}\label{eq:sharksfincond}
  d_{30} \tilde{d}_{02}^3 
+3\delta d_{21} \tilde{d}_{20} \tilde{d}_{02}^2
+ (3 d_{12} \tilde{d}_{20}^2- c_3 \tilde{d}_{20}^4)\tilde{d}_{02} 
+\delta (d_{03} - d_{11} c_3) \tilde{d}_{20}^3 
\ne0
\end{equation}
hold for both $\delta=1$ and $\delta=-1$, 
where\/ $d_{20}=\tilde{d}_{20}^2$, $d_{02}=\tilde{d}_{02}^2$.
On the other hand,
$\pi\circ f$ at\/ $0$ is a deltoid if and only if\/
$d_{20}<0$ and
\begin{equation}\label{eq:deltoidcond}
d_{30} \tilde{d}_{02}^2 -3 d_{12} \tilde{d}_{20}^2-c_3 \tilde{d}_{20}^4
\ne0\quad\text{or}\quad
d_{03} \tilde{d}_{20}^2-3 d_{21} \tilde{d}_{02}^2-d_{11} c_3 \tilde{d}_{20}^2
\ne0,
\end{equation}
where\/ $d_{20}=-\tilde{d}_{20}^2$, $d_{02}=\tilde{d}_{02}^2$.
\end{theorem}
We remark that the $\pm$-ambiguity of 
$\tilde{d}_{20}^2$, $\tilde{d}_{02}^2$ is
included by the condition \eqref{eq:sharksfincond} holding 
for both $\delta=\pm1$.
\begin{proof}
Let $\lambda$ be an identifier of singularities of $\pi\circ f$.
If $d_{20}>0$ (note that $d_{02}>0$ in \eqref{eq:whitney})
then $\lambda= -\tilde{d}_{20}^2 u^2+\tilde{d}_{02}^2 v^2$.
In this case, we set
$\eta_1=\eta_+=(\tilde{d}_{02}, \tilde{d}_{20})$,
$\eta_2=\eta_-=(-\tilde{d}_{02}, \tilde{d}_{20})$.
Then we see
\begin{align*}
\eta_\pm\eta_\pm (\pi\circ f)(0)=&
\big(2 \tilde{d}_{20}\tilde{d}_{02} ,\,
2 \tilde{d}_{20} \tilde{d}_{02} 
(d_{11}+\tilde{d}_{20}\tilde{d}_{02} )\big),\\
\eta_\pm\eta_\pm\eta_\pm (\pi\circ f)(0)=&
\big(c_3 \tilde{d}_{20}^3,\,
d_{30} \tilde{d}_{02}^3
+ 3 d_{21} \tilde{d}_{20}\tilde{d}_{02}^2
+3 d_{12} \tilde{d}_{20}^2 \tilde{d}_{02}
+d_{03} \tilde{d}_{20}^3\big).
\end{align*}
By a direct calculation, we have the assertion.
One can obtain the case of $d_{20}<0$,
under a similar calculation by setting 
$\eta_1=\eta_+=(i\tilde{d}_{02}, \tilde{d}_{20})$,
$\eta_2=\eta_-=(-i\tilde{d}_{02}, \tilde{d}_{20})$,
where $i=\sqrt{-1}$.
\end{proof}
The ${\rm SO}(2)$-normal form of $\pi\circ f$ is
(we allow $(\ep_1,\ep_2)=(-1,-1)$ in \eqref{eq:normald4})
\begin{align}
&\Bigg(uv,\ \dfrac{w_3}{2|w_1|} 
\left|\dfrac{w_2}{w_3}\right|^{1/2}u^2
+\dfrac{w_2}{2|w_1|} \left|\dfrac{w_3}{w_2}\right|^{1/2}v^2+
\dfrac{w_4}{6}u^3+\dfrac{w_5}{6}v^3\Bigg)+(O(4),O(4)),
\end{align}
where 
\begin{align*}
w_1&=
-2 ((d_{11} - \cot \theta) \cos \theta + (-d_{11}^2 + d_{02} d_{20} + d_{11} \cot \theta) \sin \theta)/d_{02},\\
w_2&=
2(d_{11} + \cot \theta - x_1) (\cot \theta \cos \theta + \sin \theta)/d_{02},\\
w_3&=
2(-d_{11} + \cot \theta + x_1) (\cot \theta \cos \theta + \sin \theta)/d_{02},
\end{align*}
and $\theta$ is an angle satisfying that
$$
\cot\theta=
\dfrac{-1+d_{11}^2-d_{02} d_{20}
+\sqrt{4 d_{11}^2+(1-d_{11}^2+d_{02} d_{20})^2}}{2 d_{11}}.
$$
One can obtain the coefficients $w_4$ and $w_5$ by 
following the procedure
of the proof of Proposition \ref{prop:normal}.
However the formula is quite complicated, and we do not
state it here.
\subsection{Planar motions}\label{sec:plmo}
As an application of the criteria, we give a concrete condition 
that a singular point appearing on a planar motion
is a sharksfin.

Let $S^1$ be the $1$-dimensional torus $S^1=\R/2\pi\Z$, and
let $SE(2)$ be the 3-dimensional Lie group
$$
SE(2)=\R^2\rtimes {\rm SO}(2)=
\left\{
(a,A)\; 
\Big|\;
A=\left(
\begin{array}{cc}
\cos \theta & -\sin\theta \\
\sin \theta & \cos \theta
\end{array}
\right),\;
a\in \R^2,
\theta\in S^1
\right\}.
$$
We take a map-germ
$\alpha\colon (\R^2,0)\to SE(2)$,
which is called a {\it planar motion-germ with\/ $2$-degrees of freedom\/} in 
the context here, see \cite{gibsonmarshxiang} for detail.
Take a point $\omega\in \R^2$,
%which is called {\it a tracing point},
and set a map-germ 
$
ev_\omega \colon (SE(2),(a_0,A_0))\rightarrow \R^2$
by
$ev_\omega(a,A)=A\omega + a$.
Then the composition 
$ev_\omega\circ\alpha\colon (\R^2,0)\rightarrow\R^2$
traces the point $w$ by the action of $\alpha(s)$, and
is called {\it a trajectory of\/  $\omega$ by\/  $\alpha$}:
$$
ev_\omega\circ\alpha(s)=A(s)\omega+a(s),
$$
where $\alpha(s)=(a(s),A(s))$.
In \cite{gibsonmarshxiang}, a generic classification of
singularities of
$ev_\omega\circ\alpha$ at $0$ 
is given when $\alpha$ at $0$
is a Whitney umbrella, namely, it is 
$\A$-equivalent to $(u,v)\mapsto(u,v^2,uv)$.
Here, we consider a special case of the above motion.
Let 
$\alpha\colon (\R,0)\rightarrow SE(2)$ and
$\beta\colon (\R,0)\rightarrow SE(2)$ be two curves.
Then the {\it composite motion\/ } 
of 
$\alpha$ and $\beta$ is defined by
$$
\nu(u,v)=\beta(v) \alpha(u)
:(\R,0)\times(\R,0)\rightarrow SE(2)
$$
where the product $\beta(v) \alpha(u)$ is that of $SE(2)$.
Composite motions are planar motions with
$2$-degrees of freedom.
In \cite{gibsonmarshxiang}, a generic classification of
singularities of $ev_\omega\circ\nu$ at $0$ is given,
where a sharksfin is in the classification.
It is also shown that a deltoid never appears on $ev_\omega\circ\nu$.
We give a concrete condition for
$ev_\omega\circ\nu$ to be a sharksfin
in terms of the geometry of $\omega$ and $\alpha, \beta$, by
using our criterion (Theorem \ref{thm:cri})
when $\alpha:(\R,0)\to(SE(2),(0,E))$ and 
$\beta:(\R,0)\to(SE(2),(0,E))$ have singular points at $0$.
We identify ${\rm SO}(2)$ with $S^1$ and we set
$\alpha(u)=((\tilde a_1(u),\tilde a_2(u)),\tilde p(u))$ and
$\beta(u)=((\tilde b_1(v),\tilde b_2(v)),\tilde q(v))$,
where $(\tilde a_1(u),\tilde a_2(u)),(\tilde b_1(v),\tilde b_2(v))\in\R^2$, 
$\tilde p(u),\tilde q(v)\in S^1$.
By the assumption, 
we can write
$\alpha(u)=((u^2a_1(u),u^2a_2(u)),u^2p(u))$ and
$\beta(v)=((v^2b_1(v),v^2b_2(v)),v^2q(v))$.
\begin{theorem}
Under the above notation,
$0$ is a corank\/  $2$ singular point of\/  $f=ev_\omega\circ\nu$.
Furthermore, $ev_\omega\circ\nu$ at\/  $0$ is\/  
a sharksfin if and only if
\begin{align}
\det\pmt{1&w_2&-w_1\\p(0)&a_1(0)&a_2(0)\\q(0)&b_1(0)&b_2(0)}\ne0,&\\
\det\pmt{1&w_2&-w_1\\p(0)&a_1(0)&a_2(0)\\p'(0)&a_1'(0)&a_2'(0)}
\ne0,&\quad
\det\pmt{1&w_2&-w_1\\q(0)&b_1(0)&b_2(0)\\q'(0)&b_1'(0)&b_2'(0)}
\ne0.
\end{align}
\end{theorem}
We remark that the curve $\alpha=(u^2a_1(u),u^2a_2(u)):
(\R,0)\to(\R^2,0)$ is a $3/2$-cusp if and only if
$$
\det\pmt{a_1(0)&a_2(0)\\a_1'(0)&a_2'(0)}\ne0.
$$

\begin{proof}
By definition,
\begin{align*}
f=& 
\Big(v^2 b_1(v)+u^2 a_1(u) \cos(v^2 q(v))
+w_1 \cos(u^2 p(u)+v^2 q(v))\\
&
-u^2 a_2(u) \sin(v^2 q(v))
-w_2 \sin(u^2 p(u)+v^2 q(v)),\\
&v^2 b_2(v)
+u^2 a_2(u) \cos(v^2 q(v))
+w_2 \cos(u^2 p(u)+v^2 q(v))\\
&+u^2 a_1(u) \sin(v^2 q(v))
+w_1 \sin(u^2 p(u)+v^2 q(v))
\Big).
\end{align*}
This shows the first assertion.
By a direct calculation, the determinant of the Jacobi matrix
of $f$ has the factor $uv$. 
Thus the Hesse matrix of $\det Jf$ at $(0,0)$ is
$$
\pmt{0&h_{11}\\h_{11}&0},
$$
where 
$$
h_{11}=
\det\pmt{a_1(0)&a_2(0)\\b_1(0)&b_2(0)}
+w_1
\det\pmt{a_1(0)&p(0)\\b_1(0)&q(0)}
+w_2
\det\pmt{a_2(0)&p(0)\\b_2(0)&q(0)}.
$$
Thus we see the solutions of Hesse quadrics 
can be taken as $\eta_1=\partial_u$ and 
$\eta_2=\partial_v$.
Since
\begin{align*}
f_{uu}(0,0)=&2(a_1(0)- w_2 p(0),a_2(0)+w_1 p(0)),\\
f_{uuu}(0,0)=&
6(a_1'(0)-w_2 p'(0),
  a_2'(0)+w_1 p'(0)),\\
f_{vv}(0,0)=&
2(b_1(0)-w_2 q(0),b_2(0)+w_1 q(0)),\\
f_{vvv}(0,0)=&
6(b_1'(0)-w_2 q'(0),
  b_2'(0)+w_1 q'(0)),
\end{align*}
and calculating 
$\det(f_{uu}(0,0),f_{uuu}(0,0))$ and
$\det(f_{vv}(0,0),f_{vvv}(0,0))$,
we have the assertion.
\end{proof}

%\section{thebibliography}

\addresslist
\end{document}